\documentclass{amsart} 
\usepackage{amsmath,amssymb,amsthm,amsfonts,amsxtra,mathtools}
\usepackage{enumerate,verbatim,mathrsfs,comment,color,url}
\usepackage[all,2cell,ps]{xy}
\usepackage[pagebackref]{hyperref}
\usepackage{graphicx}
\usepackage{tikz-cd}
\usepackage{verbatim}

\DeclareMathOperator{\Ass}{Ass}

\DeclareMathOperator{\cid}{CI-dim}
\DeclareMathOperator{\cidim}{CI-dim}

\DeclareMathOperator{\cmdim}{CM-dim}

\DeclareMathOperator{\curv}{curv}
\DeclareMathOperator{\cx}{cx}
\DeclareMathOperator{\depth}{depth}

\DeclareMathOperator{\Ext}{Ext}

\DeclareMathOperator{\gdim}{G-dim}
\DeclareMathOperator{\gcdim}{G_C-dim}
\DeclareMathOperator{\gcodim}{G_{\it C/xC}-dim}

\DeclareMathOperator{\grdim}{G_R-dim}
\DeclareMathOperator{\grade}{grade}
\DeclareMathOperator{\hdim}{H-dim}
\DeclareMathOperator{\Hom}{Hom}

\DeclareMathOperator{\id}{id}

\DeclareMathOperator{\injcx}{inj\,cx}
\DeclareMathOperator{\injcurv}{inj\,curv}

\DeclareMathOperator{\Ker}{Ker}

\DeclareMathOperator{\pd}{pd}

\DeclareMathOperator{\projcx}{proj\,cx}
\DeclareMathOperator{\projcurv}{proj\,curv}

\DeclareMathOperator{\rank}{rank}

\DeclareMathOperator{\Tor}{Tor}

\renewcommand{\ge}{\geqslant}
\renewcommand{\le}{\leqslant}

\newcommand{\fm}{\mathfrak{m}}
\newcommand{\fn}{\mathfrak{n}}

\renewcommand{\iff}{if and only if }

\theoremstyle{plain}
\newtheorem{theorem}{Theorem}[section]
\newtheorem{lemma}[theorem]{Lemma}
\newtheorem{proposition}[theorem]{Proposition}
\newtheorem{corollary}[theorem]{Corollary}

\theoremstyle{definition}
\newtheorem{definition}[theorem]{Definition}

\newtheorem{example}[theorem]{Example}

\newtheorem{para}[theorem]{}
\newtheorem{question}[theorem]{Question}

\theoremstyle{remark}
\newtheorem{remark}[theorem]{Remark}

\numberwithin{equation}{section}

\title[Homological dimensions of Burch submodules]{Homological dimensions of Burch ideals, submodules and quotients}

\author{Dipankar Ghosh}
\address{Department of Mathematics, Indian Institute of Technology Kharagpur, West Bengal - 721302, India}
\email{dipankar@maths.iitkgp.ac.in, dipug23@gmail.com}

\author{Aniruddha Saha}
\address{Department of Mathematics, Indian Institute of Technology Hyderabad, Kandi, Sangareddy - 502285, Telangana, India}
\email{ma20resch11001@iith.ac.in, sahaa43@gmail.com}

\date{November 21, 2022}

\begin{document}

\pagenumbering{arabic}
\thispagestyle{empty}

\begin{abstract}
	The notion of Burch ideals and Burch submodules were introduced (and studied) by Dao-Kobayashi-Takahashi in 2020 and Dey-Kobayashi in 2022 respectively. The aim of this article is to characterize various local rings in terms of homological invariants of Burch ideals, Burch submodules, or that of the corresponding quotients. Specific applications of our results include the following: Let $(R,\fm)$ be a commutative Noetherian local ring. Let $M=I$ be an integrally closed ideal of $R$ such that $\depth(R/I)=0$, or $M = \fm N \neq 0$ for some submodule $N$ of a finitely generated $R$-module $L$ such that either $\depth(N)\ge 1$ or $L$ is free. It is shown that: (1) $I$ has maximal projective $($resp., injective$)$ complexity and curvature. (2) $R$ is Gorenstein \iff $\Ext_R^n(M,R)=0$ for any three consecutive values of $n \ge \max\{\depth(R)-1,0\}$. (3) $R$ is CM (Cohen-Macaulay) \iff $\cmdim_R(M)$ is finite.
\end{abstract}
\maketitle

\section{Introduction}

Motivated by the results of Burch \cite{Bu68}, the notion of Burch ideals and Burch submodules were introduced (and studied) by Dao-Kobayashi-Takahashi and Dey-Kobayashi in \cite{DKT20} and \cite{DK23} respectively. These include large well-studied classes of ideals and modules. In this article, our aim is to study homological dimensions of Burch ideals, submodules and their quotients, and characterize various local rings.
%

{\it Throughout, $(R,\fm,k)$ is a commutative Noetherian local ring. All $R$-modules are assumed to be finitely generated.} Let $M$ be a Burch submodule of some $R$-module $L$, i.e., $\fm (M:_L \fm) \neq \fm M$. It follows from a result of Avramov \cite[Thm.~4]{Avr96} that $M$ has maximal projective $($resp., injective$)$ complexity and curvature. From this fact and the other existing results in the literature, in Theorem~\ref{thm:H-dim-Burch-submodules}, we observe that $R$ is H \iff $\hdim_R(M)<\infty$ \iff $\hdim_R(L/M)<\infty$, where H stands for projective, injective (or regular in case of rings), complete intersection and Gorenstein respectively. In short, these are written here as proj, inj, CI and G respectively. We prove the counterpart of these results for CM (Cohen-Macaulay) dimension, and considerably strengthen the result on G-dimension as follows.

\begin{theorem}[See Theorems~\ref{thm:CM-dim-Burch-submodules} and \ref{thm:Gor-char-vanishing-Ext} for more detailed results]\label{thm:CM-Gor-Main-results}
	Let $M$ be a Burch submodule of some $R$-module $L$.
	\begin{enumerate}[\rm (1)]
		\item When $\depth(M)\ge 1$, the ring $R$ is {\rm CM} $\Longleftrightarrow$ $\cmdim_R(M)<\infty$.
		\item When $L$ is free $($e.g., when $M=I$ is a Burch ideal of $R=L$$)$, the ring
		\begin{enumerate}[\rm (i)]
			\item $R$ is Gorenstein \iff $\Ext_R^n(M,R)=0$ for any three consecutive values of $n \ge \max\{\depth(R)-1,0\}$.
			\item $R$ is {\rm CM} $\Longleftrightarrow$ $\cmdim_R(M)<\infty$ $\Longleftrightarrow$ $\cmdim_R(L/M)<\infty$.
		\end{enumerate}
	\end{enumerate}
\end{theorem}

The technique used in the proof of the results on CM-dimension in Theorem~\ref{thm:CM-Gor-Main-results} can be applied to get simple and elementary proofs of the analogous results on other homological dimensions, see Theorem~\ref{thm:CM-dim-Burch-submodules} for a combined proof of all these results.

The main motivation for Theorems~\ref{thm:H-dim-Burch-submodules}, \ref{thm:CM-dim-Burch-submodules} and \ref{thm:Gor-char-vanishing-Ext} came from the characterizations stated below. A remarkable result due to
Burch \cite[pp~947, Cor.~3]{Bu68} states that $R$ is regular \iff projective dimension $\pd_R(I)$ is finite, where $I$ is an $\fm$-primary integrally closed ideal of $R$. The analogous results for injective and CI dimensions are shown in \cite[Cor.~6.12]{CGSZ18} and \cite[Cor.~2.7]{GP22} respectively. Moreover, in \cite[Thm.~2.6]{GP22}, it is shown that such an ideal $I$ has maximal complexity and curvature. For G-dimension, in \cite[Thm.~1.1]{CS16}, Celikbas--Sather-Wagstaff proved that $R$ is Gorenstein \iff $\gdim_R(I)$ is finite for some integrally closed ideal $I$ of $R$ such that $\depth(R/I)=0$ (a much weaker condition than $\fm$-primary).
We show that under some mild conditions, an integrally closed ideal $I$ with $\depth(R/I)=0$ is a Burch ideal, see Proposition~\ref{prop:int-closed-ideal-Burch} and Remark~\ref{rmk:DKT20-example-not-true}. Hence, as applications of Theorems~\ref{thm:H-dim-Burch-submodules}, \ref{thm:CM-dim-Burch-submodules} and \ref{thm:Gor-char-vanishing-Ext}, in Corollary~\ref{cor:characterizations-via-int-closed-ideal},
we considerably strengthen all these results. Furthermore, we obtain the analogous result for CM-dimension.

\begin{corollary}[See Corollary~\ref{cor:characterizations-via-int-closed-ideal} for more detailed results]
	Let $I$ be an integrally closed ideal of $R$ such that $\depth(R/I)=0$. Then $I$ has maximal projective $($resp., injective$)$ complexity and curvature.
	Furthermore, $R$ is Gorenstein \iff $\Ext_R^n(I,R)=0$ for any three consecutive values of $n\ge \max\{\depth(R)-1,0\}$. Moreover, $R$ is {\rm CM} \iff $\cmdim_R(I)<\infty$.	
	Particularly, it follows that
	\begin{center}
		$R$ is {\rm H} $\Longleftrightarrow$ $\hdim_R(I)<\infty$ \;$($equivalently, $\hdim_R(R/I)<\infty$$)$,
	\end{center}
	where {\rm H} denotes {\rm proj}, {\rm inj} $($regular in case of rings$)$, {\rm CI}, {\rm G} and {\rm CM} respectively.
\end{corollary}

Levin-Vasconcelos in \cite[Thm.~1.1 and remark afterward]{LV68} showed that $R$ is regular \iff $\pd_R(\fm N)<\infty$, which is equivalent to that $\id_R(\fm N)<\infty$, where $N$ is an $R$-module such that $\fm N\neq 0$. It follows from \cite[Cor.~5]{Avr96} that $R$ is CI \iff $\cidim_R(\fm N)<\infty$.
Analogous result for G-dimension can be derived from \cite[pp.~316, Lem.]{LV68} and \cite[Thm.~4.4]{CS16}.
Motivated by the result of Levin-Vasconcelos, in \cite[Thm.~1]{AP05}, Asadollahi-Puthenpurakal proved that if $N$ is an $R$-module of positive depth, then $R$ is H \iff $\hdim_R(\fm^n N)<\infty$ for some $n\ge \rho(N)$ (cf.~\cite[1.6]{AP05} for the invariant $\rho(N)$), where H can be proj, CI, G and CM. 
As other applications of Theorems~\ref{thm:H-dim-Burch-submodules}, \ref{thm:CM-dim-Burch-submodules} and \ref{thm:Gor-char-vanishing-Ext}, we combine all these results, and considerably strengthen some of them. Moreover, we obtain a few variations.

\begin{corollary}[See Corollary~\ref{cor:characterizations-via-mN}]
	Let $N$ be a submodule of an $R$-module $L$ such that $\fm N \neq 0$. Then the following statements hold true.
	\begin{enumerate}[\rm (1)]
		\item $R$ is {\rm H} \iff $\hdim_R(\fm N)<\infty$ \iff $\hdim_R(L/\fm N)<\infty$, where {\rm H} denotes {\rm proj}, {\rm inj} $($regular in case of rings$)$, {\rm CI} and {\rm G} respectively.
		\item If either $\depth(N)\ge 1$, or $L$ is free $($e.g., $N=J$ is an ideal of $R=L$$)$, then
		\begin{enumerate}[\rm (i)]
			\item $R$ is {\rm CM} \iff $\cmdim_R(\fm N)<\infty$.
			\item $R$ is Gorenstein \iff $\Ext_R^n(\fm N,R)=0$ for any three consecutive values of $n\ge \max\{\depth(R)-1,0\}$.
		\end{enumerate}
	\end{enumerate}
\end{corollary}

Now we describe in brief the contents of this article. In Section~\ref{sec:hom-inv}, we recall various homological invariants that are used in the paper. In order to prove the results on CI and CM dimensions, we need a number of lemmas, which are proved in Section~\ref{sec:G-C-dim-CM-dim}. The preliminaries on Burch submodules are discussed in Section~\ref{sec:Burch-ideals}. Particularly, Lemma~\ref{lem:Extension-of-Burch-submodule} is proved, which is crucial to obtain the results on Burch submodules of depth zero. Finally, our main results along with their applications are shown in Section~\ref{sec:Main-results-applications}.

\section{Homological invariants}\label{sec:hom-inv}

In this section, we recall the terminologies that are used in the subsequent sections. Let $M$ be an $R$-module. Let $ \beta_n^R(M) $ denote the $n$th Betti number of $M$, i.e., $ \beta_n^R(M) := \rank_k\left( \Ext_R^n(M,k) \right) $. The $i$th Bass number of $M$ is denoted by $\mu_i^R(M)$, i.e., $\mu_i^R(M):=\rank_{k}(\Ext_{R}^i(k,M))$.

\begin{para}\label{proj-inj-complexity}
\cite[Sec.~1]{Avr96} The projective complexity $\projcx_R(M)$ (resp., injective complexity $\injcx_R(M)$) is equal to a non-negative integer $b$ if $b-1$ is the smallest possible value such that there exists $\alpha>0$ satisfying $\beta_n^R(M) \le \alpha n^{b-1}$ (resp., $\mu_n^R(M) \le \alpha n^{b-1}$) for all $n\gg 0$. The projective complexity of $M$ is also denoted by $\cx_R(M)$.
\end{para}

\begin{para}\label{proj-inj-curv}
	\cite[Sec.~1]{Avr96} The projective and injective curvatures of $M$ are defined by $$\projcurv_R(M)=\limsup_{n \to \infty} \sqrt[n]{\beta_n^R(M)} \quad \mbox{and} \quad \injcurv_R(M)=\limsup_{n \to \infty} \sqrt[n]{\mu_n^R(M)}.$$
	The projective curvature of $M$ is also denoted by $\curv_R(M)$.
\end{para}

\begin{para}\label{para:global-cx-curv}
	\cite[Prop.~2]{Avr96} Note that $\max\{ \projcx_R(M) : \mbox{$M$ is an $R$-module} \} = \projcx_R(k)$. Moreover, in the above equality, $\projcx$ can be replaced by $\injcx$, $\projcurv$ and $\injcurv$ respectively. An $R$-module $M$ is said to have maximal projective (resp., injective) complexity and curvature if these invariants are same for both $M$ and $k$.
\end{para}

The notion of CI-dimension is due to Avramov-Gasharov-Peeva \cite{AGP97}.

\begin{para}\label{defn:CI-dim}
	\cite[(1.1) and (1.2)]{AGP97} A (codimension $c$) quasi-deformation of $R$ is a diagram of local ring homomorphisms $R \to R' \leftarrow Q$, where $R \to R'$ is flat, and $R' \leftarrow Q$ is a (codimension $c$) deformation, i.e., $R' \leftarrow Q$ is surjective with kernel generated by a (length $c$) regular sequence. Set $M' := M \otimes_R R'$. Define $\cid_R(0) := 0$. When $M\neq 0$, the CI-dimension of $M$ is defined to be $\cid_R(M) = $
	\[
		\inf\{ \pd_Q(M') - \pd_Q(R') : R \to R' \leftarrow Q \mbox{ is a quasi-deformation}\}.
	\]
\end{para}


The notion of semidualizing modules was introduced by Golod \cite{Go84} by the name of suitable modules. In the same paper, he also introduced the notions of $G_C$-projective modules and $G_C$-dimension.

\begin{para}\cite{Go84}
	An $R$-module $C$ is said to be semidualizing (or suitable) if
	\begin{enumerate}
		\item the natural homomorphism $R \to \Hom_R(C,C)$ is an isomorphism, and
		\item $\Ext_{R}^i(C,C)=0$ for every $i>0$.
	\end{enumerate}
\end{para}

\begin{para}\cite{Go84}
	Let $C$ be a semidualizing $R$-module. Set $(-)^{\dagger} := \Hom_R(-,C)$. An $R$-module $X$ is called $G_C$-projective if
	\begin{enumerate}
		\item the natural homomorphism $X \to X^{\dagger\dagger}$ is an isomorphism, and
		\item $ \Ext_R^i(X,C) = \Ext_R^i(X^\dagger,C) = 0 $ for every $ i > 0 $.
	\end{enumerate}
	The $G_C$-dimension of $M$, denoted by $\gcdim_R(M)$, is the least non-negative integer $n$ such that there exists an exact sequence $0 \to X_n \to \cdot\cdot\cdot \to X_1 \to X_0\to M \to 0$
	of $R$-modules such that each $X_i$ is $G_C$-projective. If such $n$ does not exist, then $\gcdim_R(M) := \infty$. The G-dimension of $M$ is defined to be  $\grdim_R(M)$, i.e., $\gcdim_R(M)$ when $C=R$, which is due to Auslander-Bridger \cite{AB69}.
\end{para}

In \cite{Ge01}, Gerko introduced the notion of CM-dimension.

\begin{para}\cite[3.1 and 3.2]{Ge01}\label{para:CM-defn-1}
	A G-quasi-deformation of $R$ is a diagram of local homomorphisms $R \to R' \leftarrow Q$, where $R \to R'$ is flat, and $R' \leftarrow Q$ is a G-deformation, i.e., a surjective homomorphism whose kernel $I$ is a G-perfect ideal of $Q$ (which means that $\gdim_Q(Q/I) = \grade(Q/I)$). Set $M' := M \otimes_R R'$. Define $\cmdim_R(0) := 0$. When $ M \neq 0 $, the CM-dimension of $M$ is defined to be $\cmdim_R(M) = $
	\[
		\inf\{ \gdim_Q(M') - \gdim_Q(R') : R \to R' \leftarrow Q \mbox{ is a G-quasi-deformation}\}.
	\]
\end{para}

The following definition of CM-dimension is equivalent to that in \ref{para:CM-defn-1}.

\begin{para}\label{para:CM-defn-2}\cite[pp.~1177, $3.2'$]{Ge01}
		If $M\neq0$, then $\cmdim_R(M)$ is
		$\inf\{\gcdim_{R'} (M \otimes_{R} R'): R \to R' $ is a local flat extension and $C$ is a semidualizing $R'$-module\}.
\end{para}

In \cite{Ge01}, Gerko gave the definitions \ref{para:CM-defn-1} and \ref{para:CM-defn-2} of CM-dimension, and mentioned that both are equivalent without giving an explicit proof. So we add a proof here.

\begin{remark}
	The definitions given in \ref{para:CM-defn-1} and \ref{para:CM-defn-2} of CM-dimension are equivalent.
\end{remark}

\begin{proof}
	Set $c_1 := \inf\{ \gdim_Q(M') - \gdim_Q(R') : R \to R' \leftarrow Q$ is a G-quasi-deformation\} and $c_2 := \inf\{\gcdim_{R'} (M \otimes_{R} R'): R \to R' $ is a local flat extension and $C$ is a semidualizing $R'$-module\}. We first prove that $c_1 \le c_2$. Consider a local flat extension $R \to R'$, and a semidualizing $R'$-module $C$ such that $\gcdim_{R'}(M \otimes_R R') < \infty$. Set $S:= R' \oplus C$, the idealization of $C$ over $R'$. Then $R \to R' \leftarrow S$ is a G-quasi deformation of $R$, and
	\begin{align*}
	\gcdim_{R'}(M \otimes_R R') = \gdim_S (M \otimes_R R') = \gdim_S (M \otimes_R R')- \gdim_S(R'),
	\end{align*}
	see \cite[Lem.~3.6]{Ge01} and the proof of \cite[Thm.~3.7]{Ge01}. It follows that $c_1 \le c_2$.
	
	In order to proof $c_1 \ge c_2$, we consider a G-quasi deformation $R \to R' \leftarrow Q$ such that $\gdim_Q(M \otimes_{R}R') < \infty$. In view of \cite[Thm.~1.4]{Ge01}, the module $C:= \Ext_{Q}^{\grade_{Q}(R')} (R',Q)$ is semidualizing over $R'$, and $\gdim_Q(M \otimes_{R}R') = \grade_{Q}(R')+ \gcdim_{R'}(M\otimes_R R')$, hence $\gdim_Q(M \otimes_{R}R') = \gdim_Q(R')+ \gcdim_{R'}(M\otimes_R R')$. This yields that $c_1 \ge c_2$. Thus $c_1=c_2$.
\end{proof}

\begin{para}\label{para:H-dim-inequalities}
	For an $R$-module $M$, by \cite[3.2]{Ge01} and \cite[(1.4)]{AGP97}, there are inequalities:
	\[
		\cmdim_R(M) \le \gdim_R(M) \le \cid_R(M) \le \pd_R(M).
	\]
	If one of these is finite, then it is equal to those to its left, which is same as $\depth(R) - \depth(M)$, see \cite[3.8]{Ge01}, \cite[(4.13.b)]{AB69}, \cite[(1.4)]{AGP97} and \cite[1.3.3]{BH93} respectively.
\end{para}

\begin{para}\label{para:char-local-rings-via-H-dim}
	The ring $R$ is H $\Longleftrightarrow$ $\hdim_R(k)<\infty$ $\Longleftrightarrow$ $\hdim_R(M)<\infty$ for every $R$-module $M$, where H can be proj, inj (regular in case of rings), CI, G and CM, see \cite[Thm.~19.2]{Mat86}, \cite[3.1.26]{BH93}, \cite[(1.3)]{AGP97}, \cite[(4.20)]{AB69} and \cite[Thm.~3.9]{Ge01} respectively.
\end{para}

\begin{para}\label{para:H-dim-syz-relations}
	For every $n\ge 0$, $\hdim_R(\Omega_n^R(M)) = \max\{ \hdim_R(M)-n, 0 \}$, where $\Omega_n^R(M)$ denotes the $n$th syzygy module of $M$. These equalities are well known when H is {\rm proj}, CI and G. See \cite[(1.9.1)]{AGP97} for CI. Similar arguments work for CM as well. The projective complexities (resp., curvatures) are same for both $\Omega_1^R(M)$ and $M$.
\end{para}


\begin{para}\label{para:H-dim-direct-sum}
	Let $M$ and $N$ be $R$-modules. Then
	\[
		\hdim_R(M \oplus N) \ge \max\{ \hdim_R(M), \hdim_R(N) \}.
	\]
	It is shown in \cite[Prop.~3.3]{AP05} for H = CM. A similar proof works for CI as well. The inequality becomes equality when H is {\rm proj}, inj and {\rm G}, see, e.g., \cite[1.3.2]{BH93}, \cite[3.1.14]{BH93} and \cite[1.2.7, 1.2.9]{Ch00} respectively. When it is equality, $\hdim$ can also be replaced by $\cx$ and $\curv$ respectively, see \cite[4.2.4.(3)]{Avr98}.
\end{para}

\section{CI and CM dimensions}\label{sec:G-C-dim-CM-dim}

Here we prove a number of lemmas on CI-dimension and CM-dimension.

   \begin{lemma}\label{lem:flat-extension}
		Let $\varphi : R \to R'$ be a flat homomorphism, and $x$ be an $R$-regular element. Then the induced map $\overline{\varphi} : R/xR \to R'/xR'$ is also flat.
	\end{lemma}
	
	\begin{proof}
		Consider a short exact sequence $0 \to M_1 \to M_2 \to M_3 \to 0 $ of $R/xR$-modules. Considering this as an exact sequence of $R$-modules, since $\varphi$ is flat, $0 \to M'_1 \to M'_2 \to M'_3 \to 0 $ is also exact, where $M'_i := M_i \otimes_R R'$. Note that
		\[
		M_i \otimes_R R' \cong (M_i \otimes_{R/xR} R/xR) \otimes_R R' \cong M_i \otimes_{R/xR} R'/xR'.
		\]
		It follows that $\overline{\varphi} : R/xR \to R'/xR'$ is flat.
	\end{proof} 

\begin{lemma}\label{lem:H-dim-M-M//xM}
	Let $x \in R$ be a regular element over both $R$ and $M$. Then the following $($in$)$equalities hold true.
	\begin{enumerate}[\rm (1)]
		\item $\cx_R(M) = \cx_{R/xR}(M/xM)$ and $\curv_R(M) = \curv_{R/xR}(M/xM)$.
		\item $\hdim_{R/xR}(M/xM) \le \hdim_R(M) $, where {\rm H} can be {\rm proj}, {\rm inj}, {\rm CI}, {\rm G} and {\rm CM}. The inequality becomes equality when {\rm H} is proj and {\rm G} respectively.
	\end{enumerate}
\end{lemma}

\begin{proof}
	The equalities in (1) are obvious as $\beta_n^R(M) = \beta_n^{R/xR}(M/xM)$ for every $n\ge 0$, see, e.g., \cite[p.~140, Lem.~2]{Mat86}. The (in)equalities in (2) are well known when H is proj, inj, {\rm CI} and {\rm G}, see \cite[1.3.5]{BH93}, \cite[3.1.15]{BH93}, \cite[(1.12.2)]{AGP97} and \cite[(4.31)]{AB69} respectively. Thus it requires to show the inequality for H = CM.
	
	If $\cmdim_R M = \infty$, then the inequality is trivial. So assume that $\cmdim_R(M)$ is finite and same as $n$. Then there exists a local flat extension $ R \to R' $ and a semidualizing $R'$-module $C$ such that $\gcdim_{R'}(M') = n $, where $M' := M \otimes_R R'$. Since $x$ is regular over both $R$ and $M$, and $ R \to R' $ is flat, $x$ is also regular over $R'$ and $M'$. Hence $C/xC$ is a semidualizing module over $R'/xR'$. Moreover, since $x$ is regular on $M'$, $\Tor^{R'}_1(R'/xR',M') = 0$. Hence, considering the natural surjection $R' \to R'/xR'$, by \cite[Prop.~6.3.1]{SSW01},
	\begin{equation}\label{eqn:GC-dim}
		\gcodim_{R'/xR'}(M' \otimes_{R'} R'/xR') = \gcdim_{R'}(M') = n.
	\end{equation}
	By Lemma~\ref{lem:flat-extension}, $R/xR \to R'/xR'$ is flat. Note that
	\begin{align*}
		M/xM \otimes_{R/xR} R'/xR' &\cong M \otimes_R R/xR \otimes_{R/xR} R'/xR' \cong M \otimes_R R'/xR' \\
		&\cong M \otimes_R R' \otimes_{R'} R'/xR' \cong M' \otimes_{R'} R'/xR'.
	\end{align*}
	So the equality \eqref{eqn:GC-dim} yields that $\cmdim_{R/xR}(M/xM) \le n = \cmdim_R(M)$. 
\end{proof}    

\begin{lemma}\label{lem:s.e.s-CM-dim}
	Let $0 \to M_1 \to M_2 \to M_3 \to 0$ be a short exact sequence of $R$-modules. If one of these $M_i$ has finite projective dimension, and another one has finite CM-dimension, then the 3rd one has finite CM-dimension.
	
	In the above mentioned result, {\rm CM} can be replaced by {\rm CI} as well.
\end{lemma}

\begin{proof}
	Consider distinct $l,m,n \in \{1,2,3\}$. Suppose $\pd_R(M_l)$ and $\cmdim_R(M_m)$ are finite. We need to show that $\cmdim_R(M_n)<\infty$. Since $\cmdim_R(M_m)<\infty$, there exists a G-quasi-deformation $R \to R' \leftarrow Q$ such that $\gdim_Q(M_m')< \infty$, where $(-)':=(-)\otimes_R R'$. Since $R\to R'$ is flat, $\pd_{R'}(M_l')< \infty$. Since $R' \leftarrow Q$ is surjective, and its kernel is a G-perfect ideal of $Q$ (cf.~\ref{para:CM-defn-1}), one obtains that $\gdim_Q(R')< \infty$. From $\pd_{R'}(M_l')< \infty$ and $\gdim_Q(R')< \infty$, applying \cite[1.2.9]{Ch00} to an $R'$-resolution of $M'$, one derives that $\gdim_Q(M_l')< \infty$. Therefore, again by \cite[1.2.9]{Ch00}, in view of the short exact sequence $0\to M_1' \to M_2'\to M_3' \to 0$, we get that $\gdim_Q(M_n')< \infty$. Thus  $\cmdim_R(M_n)<\infty$. Considering quasi-deformation (in place of G-quasi-deformation), a similar argument works for the result on CI-dimension.
\end{proof}

\begin{lemma}\label{lem:CM-direct-sum}
	Let $M$ be a submodule of a free $R$-module $F$. Let {\rm H} denotes {\rm CI} and {\rm CM} respectively. Then the following are equivalent: {\rm (1)} $\hdim_R(M \oplus F/M) <\infty$, {\rm (2)} $\hdim_R(M)<\infty$, and {\rm (3)} $\hdim_R(F/M)<\infty$. In particular,
	\[
		\hdim_R(M \oplus F/M) = \hdim_R(F/M).
	\]
\end{lemma}

\begin{proof}
	We prove the lemma for H = CM. The case H = CI can be shown in a similar way. The implications (1) $\Rightarrow$ (2) and (1) $\Rightarrow$ (3) can be obtained from \ref{para:H-dim-direct-sum}. The equivalence of (2) and (3) follows from the short exact sequence $0\to M \to F \to F/M\to 0$ and Lemma~\ref{lem:s.e.s-CM-dim}. So it is enough to prove that (2) $\Rightarrow$ (1). Let $\cmdim_R(M)<\infty$. Then there exists a G-quasi deformation $R \to R' \leftarrow Q$ such that $\gdim_Q(M')<\infty$, where $(-)':=(-)\otimes_R R'$. Since $R\to R'$ is flat, the sequence  $0\to M' \to F' \to (F/M)' \to 0$ is also exact. From the G-quasi deformation, one has that $\gdim_Q(R')< \infty$.  Therefore, by \cite[1.2.7 and 1.2.9]{Ch00}, $\gdim_Q(F/M)'< \infty$, and hence $\gdim_Q(M \oplus F/M)'< \infty$. Thus $\cmdim_R(M \oplus F/M)<\infty$. This completes the proof of the first part. For the second part, we may assume that both $\hdim_R(M \oplus F/M)$ and $ \hdim_R(F/M)$ are finite. In view of \ref{para:H-dim-inequalities}, since $\depth(M \oplus F/M) = \min\{\depth(M),\depth(F/M)\}$, one obtains that $\hdim_R(M \oplus F/M) = \max\{\hdim(M),\hdim(F/M)\} = \hdim(F/M)$. The last equality follows from \eqref{para:H-dim-syz-relations} as $M=\Omega_1^R(F/M)$.
\end{proof}

\begin{lemma}\label{lem:local-flat}
	Let $(R',\fm')$ be a Noetherian local ring, and $\varphi: R \to R'$ be a local homomorphism. Let $ \Phi : R[X]_{\langle \fm, X \rangle} \to R'[X]_{\langle \fm',X \rangle}$ be the map induced by $\varphi$. 
	\begin{enumerate}[\rm (1)]
		\item If $\varphi$ is flat $($resp., local, surjective$)$, then so is $\Phi$.
		\item $\Ker(\Phi) = \langle \Ker(\varphi) \rangle$.
		\item Every quasi-deformation $R \to R' \leftarrow Q$ induces another quasi-deformation $  R[X]_{\langle \fm, X \rangle} \to R'[X]_{\langle \fm',X \rangle} \leftarrow Q[X]_{\langle \fn, X \rangle}$, where $\fn$ is the maximal ideal of $Q$.
	\end{enumerate}
\end{lemma}  

\begin{proof}
	(1) If $\varphi$ is local (resp., surjective), then by the construction, $\Phi$ is so. Suppose $\varphi$ is flat. Then, since $R'[X]\otimes_{R[X]}(-)\cong R'\otimes_R R[X]\otimes_{R[X]}(-)$, the induced map $R[X] \to R'[X] $ is also flat. Since localization is flat, and the composition of two flat homomorphisms is flat, the homomorphism $R[X] \to R'[X]_{\langle \fm',X\rangle}$ is flat. Therefore, for any module $L$ over $R[X]_{\langle \fm,X\rangle}$, in view of the isomorphisms
	\begin{align*}
		&R'[X]_{\langle \fm',X \rangle} \otimes_{R[X]_{\langle \fm,X\rangle}} L \\
		& \cong \left( R'[X]_{\langle \fm',X \rangle} \otimes_{R[X]} R[X]_{\langle \fm,X\rangle} \right) \otimes_{R[X]_{\langle \fm,X\rangle}} \left( R[X]_{\langle \fm,X\rangle} \otimes_{R[X]} L \right) \\
		& \cong R'[X]_{\langle \fm',X \rangle} \otimes_{R[X]} R[X]_{\langle\fm,X\rangle} \otimes_{R[X]} L,
	\end{align*}
	one obtains that $\Phi$ is flat.
	
	(2) The containment $\langle \Ker(\varphi) \rangle \subseteq \Ker(\Phi)$ is trivial. For other containment, consider an element $y= (a_0+a_1X+a_2 X^2+ \cdots +a_m X^m)/1$ of $\Ker(\Phi)$. Then $\phi(a_0)+\phi(a_1)X+\cdots +\phi(a_m)X^m =0$ in $R'[X]_{\langle \fm',X\rangle}$. Therefore, by forward induction on $i=0,1,\dots,m$, one concludes that $t\varphi(a_i)=0$ for some $t \in R'\smallsetminus \fm'$, and hence $\varphi(a_i)=0$ for every $i=0,1,\dots,m$. Thus $y \in \langle \Ker(\varphi) \rangle$.
	
	(3) Note that $Q\to Q[X]_{\langle \fn, X \rangle}$ is local flat. Hence every $Q$-regular sequence is also $Q[X]_{\langle \fn, X \rangle}$-regular. Therefore, in view of the definition of quasi-deformation \ref{defn:CI-dim}, the statement (3) follows from (1) and (2).
\end{proof}

\begin{lemma}\label{lem:CM-dim-relation-mod-x-R-S}
	Let $S=R[X]_{\langle \fm, X \rangle}$. Let $M$ be an $S$-module such that $X$ is $M$-regular. Then $\hdim_S(M) \le \hdim_R(M/XM)$ where $\rm H$ is $\rm CM$ and $\rm CI$ respectively.
\end{lemma}

\begin{proof}
	Consider the case when $\rm H=CM$. We may assume that $\cmdim_R(M/XM)$ is finite, say $ n $. Then there exists a local flat homomorphism $R \to R'$ and a semidualizing $R'$-module $C$ such that $\gcdim_{R'}(M/XM \otimes_{R}R')=n$. Let $\fm'$ be the maximal ideal of $R'$. Set $S' := R'[X]_{\langle \fm',X \rangle}$. Being composition of two flat extensions $R' \to R'[X]$ and $R'[X] \to S'$, the extension $R' \to S'$ is flat. Therefore, for any two $R'$-modules $U$ and $V$,
		\[
		 \Ext_{R'}^i(U,V) \otimes_{R'}S'  \cong \Ext_{S'}^i( U \otimes_{R'}S',V \otimes_{R'}S') \quad \mbox{for every } i \ge 0.
		\]
	It follows that $C \otimes_{R'}S'$ is a semidualizing module over $S'$. Moreover, for every $G_C$-projective $R'$-module $L$, the module $L\otimes_{R'}S'$ is $G_{C \otimes_{R'}S'}$-projective over $S'$. Note that $R' \to S'$ is flat and local, hence faithfully flat. Since $\gcdim_{R'}(M/XM \otimes_{R}R') = n $, there exists a $G_C$-projective resolution of $(M/XM \otimes_{R}R')$ over $R'$ of length $n$. Applying $(-)\otimes_{R'}S'$ to this resolution, one obtains a $G_{C \otimes_{R'}S'}$-projective resolution of $ M/XM \otimes_{R}R'\otimes_{R'}S'$ over $S'$ of length $n$.
	Thus, since
		\begin{align*}
			M/XM \otimes_{R} R'\otimes_{R'}S' &\cong (M\otimes_{S} S/XS)\otimes_{R}R' \otimes_{R'}S' \\ &\cong M\otimes_{S}(R \otimes_{R}R')\otimes_{R'}S'  \\ &\cong M \otimes_{S} (R' \otimes_{R'}S') \cong M\otimes_{S} S',
		\end{align*}	
	it follows that ${\rm G}_{(C \otimes_{R'}S')}\mbox{-}{\rm dim}_{S'} ( M\otimes_{S} S') \le n$. Therefore, since $S \to S'$ is a local flat homomorphism (by Lemma~\ref{lem:local-flat}), one gets that $\cmdim_S(M) \le n$.
	
	It remains to prove the statement when $\rm H=CI$. Let $\cid_R(M/XM)=n < \infty$. Then there exists a quasi-deformation $R \stackrel{\varphi}\longrightarrow R' \stackrel{\psi}\longleftarrow Q$ of $R$ such that $\pd_Q(M/XM \otimes_{R}R')=n+\pd_Q(R')$. By Lemma~\ref{lem:local-flat},  $S \stackrel{\Phi}\longrightarrow S' \stackrel{\Psi}\longleftarrow T$ is a quasi-deformation of $S$, and $(\Ker(\psi))T = \Ker(\Psi)$, where $S'=R'[X]_{\langle \fm', X \rangle}$ and $T=Q[X]_{\langle \fn, X \rangle}$, and $\fm'$ and $\fn$ are the maximal ideals of $R'$ and $Q$ respectively. Since 
	\begin{align*}
		M/XM \otimes_{R} R' \cong M \otimes_{S} S/XS \otimes_{R} R' \cong M \otimes_{S} R \otimes_{R} R' \cong M \otimes_{S} R',
	\end{align*}
	it follows that $\pd_Q(M \otimes_{S}R')=n+\pd_Q(R')$. Therefore, in view of
	\begin{align*}
		(M \otimes_{S} R') \otimes_{Q} T &\cong M \otimes_{S} (R' \otimes_{Q} T) \cong M \otimes_{S}(Q/\Ker(\psi) \otimes_{Q}T)\\
		& \cong M \otimes_{S}T/(\Ker(\psi))T \cong M \otimes_{S}T/\Ker(\Psi) \cong M \otimes_{S}S',
	\end{align*}
	since $Q \to T$ is flat and local, one obtains that $\pd_{T}(M \otimes_{S}S')=n+\pd_{T}(S')$. Hence $\cid_{S}(M) \le n=\cid_R(M/XM)$.
\end{proof}

\section{Burch submodules and their quotients}\label{sec:Burch-ideals}

The notion of $\fm$-full ideals was introduced by D.~Rees (unpublished), while weakly $\fm$-full and Burch ideals were introduced in \cite[3.7]{CIST18} and \cite[2.1]{DKT20} respectively. These are extended to submodules in \cite[2.1]{AP05}, \cite[4.1]{DK23} and \cite[3.1]{DK23} respectively.

\begin{definition}\cite[Def.~3.1]{DK23}\label{defn:Burch-submodule}
	Let $L$ be an $R$-module. A submodule $M$ of $L$ is said to be Burch if $\fm(M:_L \fm) \neq \fm M$, i.e., $\fm(M:_L \fm) \not\subseteq \fm M$. A Burch submodule of $R$ is called a Burch ideal of $R$, which is due to \cite[Def.~2.1]{DKT20}.
\end{definition}   

\begin{remark}\label{rmk:Burch-submodules}
	Let $M$ be a Burch submodule of some $R$-module $L$. Then, by the definition, $M\neq 0$ and $L/M\neq 0$. Moreover, by \cite[Lem.~3.3]{DK23}, $\depth(L/M)=0$.
\end{remark}

\begin{definition}(\cite[2.1]{AP05} and \cite[4.1]{DK23})
	A submodule $M$ of $L$ is called:
	\begin{enumerate}[\rm (1)]
		\item $\fm$-full if $(\fm M:_L x)=M$ for some $x\in \fm$.
		\item weakly $\fm$-full if $(\fm M:_L \fm)=M$.
	\end{enumerate}
\end{definition}

\begin{remark}\label{rmk:strata}
	 \cite[pp.~13]{DK23}
	A submodule $M$ of an $R$-module $L$ is
	\[
	\mbox{$\fm$-full} \;\; \Longrightarrow \;\; \mbox{weakly $\fm$-full} \;\; \stackrel{(*)}{\Longrightarrow} \;\; \mbox{Burch},
	\]
	where $\depth(L/M)$ is assumed to be zero for $(*)$.
\end{remark}

\begin{example}\label{exam:Burch-submodules-quotient}
	\cite[3.5]{DK23} If $N$ is a submodule of some $R$-module $L$ such that $\fm N \neq 0$, then $M:=\fm N$ is a Burch submodule of $L$.
\end{example}

We show that a large class of integrally closed ideals are Burch ideals under some mild conditions on the ring.

\begin{proposition}\label{prop:int-closed-ideal-Burch}
	Suppose that $k$ is infinite, and $R$ is not a field. Let $I$ be an integrally closed ideal of $R$ such that $\depth(R/I)=0$. Then the ideal $I$ is Burch.
\end{proposition}

\begin{proof}
	By virtue of \cite[Thm.~2.4]{Go87}, either $I$ is $\fm$-full or $I=\sqrt{0}$. If $I$ is $\fm$-full, since $\depth(R/I)=0$, by \cite[Cor.~2.4]{DKT20}, $I$ is Burch. So we may assume that $I = \sqrt{0}$. Then $I^n = 0 $ for some $n\ge 1$. Since $\depth(R/I)=0$, there exists $x \notin I$ such that $\fm x \subseteq I$. If $x\in \fm$, then $x^2\in I$, hence $x^{2n}=0$, which implies that $x \in \overline{I}=I$, a contradiction. Thus $x$ is a unit. It follows from $\fm x \subseteq I$ that $I=\fm$, which is also a Burch ideal of $R$ (by \cite[2.2.(2)]{DKT20}).
\end{proof}

\begin{remark}\label{rmk:DKT20-example-not-true}
	The statement given in \cite[Ex.~2.2.(4)]{DKT20} is not true in general. Indeed, the condition that $\depth(R/I)=0$ in Proposition~\ref{prop:int-closed-ideal-Burch} cannot be omitted as remarked in \ref{rmk:Burch-submodules} even when $\depth(R)\ge 1$. Consider $R=k[[x,y,z,w]]/(xy,yz,zx,w^2)$ over an (infinite) field $k$. Then $R$ is a CM local ring of dimension $1$. Set $I:=(w)$. Then $I$ is an integrally closed ideal of $R$ and $\depth(R/I)=1$. Since $\fm(I :_R \fm) = \fm I$, the ideal $I$ is not Burch.
\end{remark}

The following lemma is crucial in order to prove our results on Burch submodules of depth zero.

\begin{lemma}\label{lem:Extension-of-Burch-submodule}
	Let $S=R[X]_{\langle \fm, X \rangle}$. Let $M$ be a Burch submodule of an $R$-module $L$. Set $L':=L[X]_{\langle \fm, X \rangle}$ and $M':=(M+XL[X])_{\langle \fm, X \rangle}$. Then 
	\begin{enumerate} [\rm (1)] 
		\item $R$ is regular $($resp., {\rm CI}, Gorenstein and {\rm CM}$)$ \iff $S$ is so.
		\item $M'$ is a Burch submodule of the $S$-module $L'$.
		\item $M'/XM' \cong M \oplus L/M$ as $R$-modules.
		\item When $L$ is free, the following equalities hold true.
		\begin{enumerate}[\rm (i)]
			\item $\hdim_R(L/M)=\hdim_S(M')$, where {\rm H} can be {\rm proj}, {\rm CI}, {\rm G} and {\rm CM}.
			\item $\cx_R(M)=\cx_S(M')$ and $\curv_R(M)=\curv_S(M')$.
		\end{enumerate}
		\item $\cx_R(k)=\cx_S(k)$ and $\curv_R(k)=\curv_S(k)$.
	\end{enumerate}	
\end{lemma}

\begin{proof}
	(1) Note that the maximal ideal of $S$ is given by $\fn= \fm S + XS$. The assertions in (1) follow from the fact that $X\in\fn\smallsetminus\fn^2$ is $S$-regular.
	
	(2) Since $\fm(M:_L \fm) \neq \fm M$, there exist $a \in \fm$ and $y \in (M:_L \fm)$ such that $ay \notin \fm M $. Then $ \frac{a}{1} \in \fn$ and $\frac{y}{1} \in (M':_{L'}\fn)$. Since $a \in \fm$, $y \in L$ and  $ay \notin \fm M$, it follows that $\frac{a}{1}\cdot \frac{y}{1} \notin \fn M'$. (Indeed, if $\frac{ay}{1} \in \fn M'$, then $say\in \langle \fm, X \rangle (M+XL[X])$ for some $s\in R[X]\smallsetminus \langle \fm, X \rangle$, and hence comparing the terms, one obtains that $ay\in \fm M$, a contradiction). Therefore $\fn(M':_{L'} \fn)\neq \fn M'$.
	
	(3) As $R$-modules, we have the following isomorphisms:
	\begin{equation}\label{M'-isomorphism-with-direct-sum}
		\dfrac{M+XL[X]}{X(M+XL[X])} \cong \dfrac{M \oplus XL \oplus X^2L[X]}{0\oplus XM \oplus X^2 L[X]} \cong M \oplus \frac{L}{M}.
	\end{equation}
	Localizing \eqref{M'-isomorphism-with-direct-sum} with respect to $T:= R[X] \smallsetminus {\langle \fm, X \rangle}$, we get the desired isomorphism $M'/XM' \cong M \oplus L/M$ as $R$-modules.

	(4) Note that $X$ is regular over both $S$ and $M'$. Therefore
	\begin{align*}
		\hdim_S(M') &=\hdim_{S/XS}(M'/XM') \quad\mbox{[by Lemmas~\ref{lem:H-dim-M-M//xM} and \ref{lem:CM-dim-relation-mod-x-R-S}]} \\
		&=\hdim_R(M \oplus L/M) \quad\mbox{[by (3)]}\\
		&= \max\{\hdim_R(M), \hdim_R(L/M)\} \quad \mbox{[by \ref{para:H-dim-direct-sum}]}\\
		&=\hdim_R(L/M) \quad\mbox{[by \ref{para:H-dim-syz-relations} as $L$ is free]}
	\end{align*}
	when H is proj and G respectively. The same equalities hold true for complexity and curvature as well. When {\rm H} is {\rm CI} and {\rm CM}, for the last two equalities, by Lemma~\ref{lem:CM-direct-sum}, one directly has $\hdim_R(M \oplus L/M) = \hdim_R(L/M)$.
   
   (5) These equalities are shown in \cite[Lem.~3.3]{DG22}.
\end{proof}

\begin{para}\label{para:k-summand-M-mod-xM}
	Let $M$ be a Burch submodule of some $R$-module $L$. It is shown in \cite[Lem.~3.6]{DK23} that if $\depth(M)\ge 1$, then there exists an $M$-regular element $x \in \fm$ such that $k$ is a direct summand of $M/xM$. Note that this element $x$ can be chosen in such a way that $x\notin\fm^2$. Indeed, in the proof of \cite[Lem.~3.6]{DK23}, we need to use the (general version of) prime avoidance lemma by avoiding the prime ideals in $\Ass_R(M)$ and two (possibly non-prime) ideals $\fm^2$ and $(\fm M :_R (M:_L\fm) )$. Moreover, in addition, if $\depth(R)\ge 1$, then $x$ can be chosen to be $R$-regular as well.
\end{para}

\begin{para}\cite{DK23}\label{Dey-Kobayashi-results} 
	Let $M$ be a Burch submodule of some $R$-module $L$. Let $n\ge 1$ be an integer. For an $R$-module $N$, the following statements hold true.
	\begin{enumerate}[\rm (1)]
		\item \cite[Prop.~3.16, Cor.~3.14 and Prop.~3.15]{DK23} $\pd_R(N)<n$ if at least one of the following conditions holds true.
		\begin{enumerate}[\rm (i)]
			\item $\Tor^R_n(N,M) = \Tor^R_{n-1}(N,M) = 0$.
			\item $\Ext_R^n(N,M) = \Ext_R^{n+1}(N,M) = 0$.
			\item $\Tor^R_n(N,L/M) = \Tor^R_{n+1}(N,L/M) = 0$.
			\item $\Ext_R^{n}(N,L/M) = \Ext_R^{n-1}(N,L/M) = 0$.
		\end{enumerate}
		\item \cite[Prop.~3.18]{DK23} $\id_R(N)<\infty$ if $\depth(M)\ge 1$ and $\Ext_R^i(M,N)=0$ for any two consecutive values of $i\ge \depth(N)-1$.
	\end{enumerate}
\end{para}

In the next section, we provide a variation of \cite[Prop.~3.18]{DK23}, see Theorem~\ref{thm:Gor-char-vanishing-Ext}.

\section{Main results and applications}\label{sec:Main-results-applications}

We start by observing the following characterizations of various local rings in terms of Burch submodules and their quotients from the existing results in the literature. The statement \ref{thm:H-dim-Burch-submodules}.(1) is noted in \cite[2.5]{DKT20} and \cite[3.11]{DK23} for Burch ideals and submodules respectively only for projective complexity and curvature.

\begin{theorem}\label{thm:H-dim-Burch-submodules}
	Let $M$ be a Burch submodule of some $R$-module $L$. Then
	\begin{enumerate}[\rm (1)]
		\item $M$ has maximal projective $($resp., injective$)$ complexity and curvature.
		\item $R$ is regular
		$\Longleftrightarrow$ $\Ext_R^n(M,N) = \Ext_R^{n+1}(M,N)=0$ for some $n\ge 1$, where $N$ is a Burch submodule of some $R$-module. Under these conditions, $\pd_R(M)<n$.
		\item $\pd_R(L/M)<\infty$ $\Longleftrightarrow$ $R$ is regular $\Longleftrightarrow$ $\id_R(L/M)<\infty$.
		\item The following are equivalent: {\rm (i)} $R$ is complete intersection, {\rm (ii)} $\cidim_R(M)<\infty$,  {\rm (iii)} $\cidim_R(L/M)<\infty$, {\rm (iv)} $\projcx_R(M)<\infty$, {\rm (v)} $\injcx_R(M)<\infty$, {\rm (vi)} $\projcurv_R(M)\le 1$ and {\rm (vii)} $\injcurv_R(M)\le 1$.
		\item $R$ is Gorenstein $\Longleftrightarrow$ $\Ext_R^n(M,R)=0$ for all $n\gg 0$ $\Longleftrightarrow$ $\Ext_R^n(L/M,R)=0$ for all $n\gg 0$.
		\item All together, $R$ is {\rm H} $\Longleftrightarrow$ $\hdim_R(M)<\infty$ $\Longleftrightarrow$ $\hdim_R(L/M)<\infty$, where {\rm H} denotes {\rm proj}, {\rm inj} $($regular in case of rings$)$, {\rm CI} and {\rm G} respectively.
	\end{enumerate}
\end{theorem}

\begin{proof}
	Since $M$ is a submodule of $(M:_L \fm)$ satisfying $M \supseteq \fm (M:_L \fm) \neq \fm M$, by virtue of \cite[Thm.~4]{Avr96}, $M$ has maximal projective $($resp., injective$)$ complexity and curvature. For (2), if $\Ext_R^n(M,N) = \Ext_R^{n+1}(M,N)=0$, then by \ref{Dey-Kobayashi-results}.(1).(ii), $\pd_R(M)<n$, which implies that $\cx_R(k)=\cx_R(M)=0$, i.e., $\pd_R(k)<\infty$, hence $R$ is regular. The equivalences in (3) can be deduced from \ref{Dey-Kobayashi-results}.(1).(iii) and (iv) respectively. The statements in (4) (except (4).(i) $\Leftrightarrow$ (iii)) are consequences of (1), see, e.g., \cite[pp.~321 and Thm.~3]{Avr96}. In view of \ref{Dey-Kobayashi-results}.(1).(i) and (iii), $M$ and $L/M$ are test $R$-modules. Hence the equivalences in (5) and (4).(i) $\Leftrightarrow$ (iii) can be obtained from \cite[Thm.~4.4]{CS16} and \cite[Thm.~3.4.4]{Ta19} respectively. The equivalences in (6) are consequences of (2), (3), (4) and (5).
\end{proof}

The result on CM-dimension in the theorem below is totally new, while the results on other homological invariants can be observed from Theorem~\ref{thm:H-dim-Burch-submodules}. However, we give a simple, elementary and combined proof of the results. Note that Theorem~\ref{thm:H-dim-Burch-submodules}.(1) is shown by applying \cite[Thm.~4]{Avr96}, where cohomology representation of the homotopy Lie algebra is used in the proof of \cite[Thm.~4]{Avr96}.

\begin{theorem}\label{thm:CM-dim-Burch-submodules}
	Let $M$ be a Burch submodule of some $R$-module $L$. Suppose either $\depth(M)\ge 1$, or $L$ is free $($e.g., $M=I$ is a Burch ideal of $R=L$$)$. Then the following statements hold true.
	\begin{enumerate}[\rm (1)]
		\item The ring $R$ is {\rm H} \iff $\hdim_R(M)<\infty$, where {\rm H} denotes {\rm proj} $($regular in case of rings$)$, {\rm CI}, {\rm G} and {\rm CM} respectively.
		\item $\cx_R(M) = \cx_R(k)$ and $\curv_R(M) = \curv_R(k)$.
	\end{enumerate}
\end{theorem}

\begin{proof}
	The `only if' parts in (1) are well known, see \ref{para:char-local-rings-via-H-dim}. So we need to prove (2) and the `if' parts in (1). When $L$ is free, note that $\hdim_R(M)<\infty$ \iff $\hdim_R(L/M)<\infty$ (Lemma~\ref{lem:CM-direct-sum}). So, in the case when $L$ is free, by virtue of Lemma~\ref{lem:Extension-of-Burch-submodule}, we may replace $R$ and $M$ by $S:=R[X]_{\langle \fm, X \rangle}$ and $M':=(M+XL[X])_{\langle \fm, X \rangle}$ respectively, and assume that $\depth(R)\ge 1$ and $\depth(M)\ge 1$. Then, in view of \ref{para:k-summand-M-mod-xM}, there exists an element $x \in \fm\smallsetminus\fm^2$ which is regular over both $R$ and $M$ such that $k$ is a direct summand of $M/xM$ as an $R/xR$-module. Hence 
  \begin{align*}
  	\hdim_{R/xR}(k) & \le \hdim_{R/xR}(M/xM) \quad \mbox{[by \ref{para:H-dim-direct-sum}]} \\ & \le \hdim_R(M) \quad \mbox{[by Lemma~\ref{lem:H-dim-M-M//xM}].}
  \end{align*} 
   Thus, if $\hdim_R(M) < \infty$, then $\hdim_{R/xR}(k) < \infty$, which (in view of \ref{para:char-local-rings-via-H-dim}) implies that $R/xR$ is H, and hence $R$ is H. For (2), note that
   \[
   		\cx_R(k) = \cx_{R/xR}(k) \le \cx_{R/xR}(M/xM) = \cx_R(M) \le \cx_R(k),
   \]
   where the respective (in)equalities are obtained from \cite[Lem.~3.3]{DG22}, \ref{para:H-dim-direct-sum}, Lemma~\ref{lem:H-dim-M-M//xM} and \ref{para:global-cx-curv}. Hence $\cx_R(M) = \cx_R(k)$. Similarly, $\curv_R(M) = \curv_R(k)$.
\end{proof}

The next result shows that, by virtue of Lemma~\ref{lem:Extension-of-Burch-submodule}, the depth condition on the Burch submodule $M$ in \cite[Prop.~3.18]{DK23} can be removed under some extra vanishing of certain Ext modules.

\begin{theorem}\label{thm:Gor-char-vanishing-Ext}
	Let $M$ be a Burch submodule of some $R$-module $L$. Let $N$ be an $R$-module. Let $n\ge \depth(N)$ be such that
	\[
	\Ext_R^i(M,N) = \Ext_R^j(L,N) = 0 \quad \mbox{for all }i = n-1, n, n+1\mbox{ and } j = n, n+1.
	\]
	Then $\id_R(N)<\infty$.
\end{theorem} 

\begin{proof}
	Considering the long exact sequence 
	\[
	\cdot \cdot \cdot \to \Ext_R^{l-1}(M,N) \to \Ext_R^{l}(L/M,N) \to \Ext_R^{l}(L,N) \to \Ext_R^{l}(M,N) \to \cdot \cdot\cdot,
	\]
	from the given hypotheses, one obtains that $\Ext_R^j(L/M,N)=0$ for both $j=n,n+1$. Therefore, by virtue of Lemma~\ref{lem:Extension-of-Burch-submodule}.(3), $\Ext_R^i(M'/XM',N)=0$ for both $i=n,n+1$, where $M':=(M+XL[X])_{\langle \fm, X \rangle} $, a Burch submodule of $L':=L[X]_{\langle \fm, X \rangle}$ over  $S:=R[X]_{\langle \fm, X \rangle}$. Set  $N':=N[X]_{\langle \fm, X \rangle}$. Then, by \cite[3.1.16]{BH93},
	\begin{align*}
	\Ext_S^{i+1}(M'/XM',N') \cong \Ext_{S/XS}^{i}(M'/XM',N'/XN')=\Ext_R^{i}(M'/XM',N)=0
	\end{align*}
	for $i=n,n+1$. Hence, in view of the exact sequences
	\[
	\Ext_S^{i}(M',N') \stackrel{X}{\longrightarrow} \Ext_S^{i}(M',N') \longrightarrow \Ext_S^{i+1}(M'/XM',N'),
	\]
	the Nakayama Lemma yeilds that $\Ext_S^{i}(M',N')=0$ for $i=n,n+1$. Note that $\depth_S(M')\ge 1$ and $n\ge \depth_R(N)=\depth_S(N')-1$. Therefore, by \ref{Dey-Kobayashi-results}.(2), $\id_S(N')<\infty$, which implies that $\id_{S/XS}(N'/XN')< \infty$ (cf.~\cite[3.1.15]{BH93}), i.e., $\id_R(N)<\infty$.
\end{proof}

As an immediate consequence of Theorem~\ref{thm:Gor-char-vanishing-Ext}, one deduces the following.

\begin{corollary}\label{cor:char-Gor-Burch-ideals}
	Let $M$ be a Burch submodule of some free $R$-module $($e.g., $M=I$ is a Burch ideal of $R$$)$. If $\Ext_R^m(M,R)=0$ for any three consecutive values of $m \ge \max\{\depth(R)-1,0\}$, then $R$ is Gorenstein.
\end{corollary}

\begin{proof}
	In Theorem~\ref{thm:Gor-char-vanishing-Ext}, consider $L$ as a free $R$-module, and $N=R$. Let $m \ge \max\{\depth(R)-1,0\}$. Then $m \ge \depth(R)-1$ and $m+1\ge 1$. Hence $\Ext_R^j(L,R) = 0$ for all $j\ge m+1$. Therefore, if $\Ext_R^i(M,R)=0$ for $i = m, m+1, m+2$, by Theorem~\ref{thm:Gor-char-vanishing-Ext}, $\id_R(R)<\infty$, i.e., $R$ is Gorenstein.
\end{proof}

In view of Proposition~\ref{prop:int-closed-ideal-Burch} and Example~\ref{exam:Burch-submodules-quotient}, Theorems~\ref{thm:H-dim-Burch-submodules}, \ref{thm:CM-dim-Burch-submodules} and \ref{thm:Gor-char-vanishing-Ext} have many applications. For example, an integrally closed ideal $I$ of $R$ with $\depth(R/I)=0$ can be used to characterize various local rings.

\begin{corollary}\label{cor:characterizations-via-int-closed-ideal}
	Let $I$ be a nonzero ideal of $R$ such that $\depth(R/I)=0$. Assume that $I$ is integrally closed or weakly $\fm$-full or Burch. Then:
	\begin{enumerate}[\rm (1)]
		\item $I$ has maximal projective $($resp., injective$)$ complexity and curvature.
		\item $R$ is regular
		$\Longleftrightarrow$ $\Ext_R^n(I,J) = \Ext_R^{n+1}(I,J)=0$ for some $n\ge 1$, where $J\neq 0$ is an ideal satisfying $\depth(R/J)=0$, and $J$ is integrally closed or weakly $\fm$-full or Burch. Under these conditions, $\pd_R(I)<n$.
		\item $R$ is complete intersection $\Longleftrightarrow$ $\cidim_R(I)<\infty$ $\Longleftrightarrow$ $\projcx_R(I)<\infty$ $\Longleftrightarrow$ $\injcx_R(I)<\infty$ $\Longleftrightarrow$ $\projcurv_R(I)\le 1$ $\Longleftrightarrow$ $\injcurv_R(I)\le 1$.
		\item $R$ is Gorenstein $\Longleftrightarrow$ $\Ext_R^n(I,R)=0$ for any three consecutive values of $n\ge \max\{\depth(R)-1,0\}$. We need two consecutive vanishing when $\depth(R)\ge 1$.
		\item $R$ is {\rm CM} $\Longleftrightarrow$ $\cmdim_R(I)<\infty$.
	\end{enumerate}
\end{corollary}

\begin{proof}
	If the residue field is finite, then passing through $R[X]_{\fm R[X]}$, we may assume that $R$ has infinite residue field $k$. If $R$ is a field, then there is nothing to prove. So we may assume that $\fm\neq 0$. Therefore the proof follows from Proposition~\ref{prop:int-closed-ideal-Burch} and Remark~\ref{rmk:strata}, by applying Theorem~\ref{thm:H-dim-Burch-submodules}.(1), (2) and (4), Corollary~\ref{cor:char-Gor-Burch-ideals} and Theorem~\ref{thm:CM-dim-Burch-submodules} respectively. For the second part of (4), we must use \ref{Dey-Kobayashi-results}.(2) and the fact that $\depth(I)\ge 1$ if $\depth(R)\ge 1$.
\end{proof}

\begin{remark}
	\begin{enumerate}[(1)]
		\item 
		The statements (1), (2) and (3) in Corollary~\ref{cor:characterizations-via-int-closed-ideal} considerably strengthen the results \cite[Thm.~2.6]{GP22}, \cite[pp~947, Cor.~3]{Bu68}, \cite[6.12]{CGSZ18} and \cite[Cor.~2.7]{GP22}. Moreover, Corollary~\ref{cor:characterizations-via-int-closed-ideal}.(2) notably improves \cite[Cor.~3.14]{CIST18}, the main outcome of \cite{CIST18}.
		\item 
		When $R$ is CM and $I$ is as in Corollary~\ref{cor:characterizations-via-int-closed-ideal},	it is shown in \cite[Cor.~3.16]{CIST18} that $R$ is Gorenstein if $\Ext_R^i(I,R)=0$ for any $\dim(R)+2$ consecutive values of $i \ge 1$. Corollary~\ref{cor:characterizations-via-int-closed-ideal}.(4) considerably strengthen this result when $\dim(R) \ge 1$.
	\end{enumerate}
\end{remark}

\begin{corollary}\label{cor:characterizations-via-mN}
	Let $N$ be a submodule of an $R$-module $L$ such that $\fm N \neq 0$.
	\begin{enumerate}[\rm (1)]
		\item Let {\rm H} denotes {\rm proj}, {\rm inj} $($regular in case of rings$)$, {\rm CI} and {\rm G} respectively. Then
		\begin{center}
			$R$ is {\rm H} $\Longleftrightarrow$ $\hdim_R(\fm N)<\infty$ $\Longleftrightarrow$ $\hdim_R(L/\fm N)<\infty$.
		\end{center}
		\item If either $\depth(N)\ge 1$, or $L$ is free $($e.g., $N=J$ is an ideal of $R=L$$)$, then
		\begin{enumerate}[\rm (i)]
			\item $R$ is {\rm CM} $\Longleftrightarrow$ $\cmdim_R(\fm N)<\infty$.
			\item $R$ is Gorenstein $\Longleftrightarrow$ $\Ext_R^n(\fm N,R)=0$ for any three consecutive values of $n\ge \max\{\depth(R)-1,0\}$. We need two consecutive vanishing if $\depth(N)\ge 1$.
		\end{enumerate}
	\end{enumerate}
\end{corollary}

\begin{proof}
	In view of Example~\ref{exam:Burch-submodules-quotient}, $\fm N$ is a Burch submodule of $L$. Hence the equivalences in (1) and (2).(i) can be obtained from Theorems~\ref{thm:H-dim-Burch-submodules}.(6) and \ref{thm:CM-dim-Burch-submodules}.(1) respectively, while (2).(ii) follows from \ref{Dey-Kobayashi-results}.(2) and Corollary~\ref{cor:char-Gor-Burch-ideals}.
\end{proof}

In view of Theorems~\ref{thm:H-dim-Burch-submodules} and \ref{thm:CM-dim-Burch-submodules}, the following questions arise naturally.

\begin{question}
	Let $M$ be a Burch submodule of some $R$-module $L$.
	\begin{enumerate}[\rm (1)]
		\item Does $L/M$ have maximal projective $($resp., injective$)$ complexity and curvature?
		\item If $\cmdim_R(L/M)<\infty$, then is $R$ {\rm CM}?
	\end{enumerate}
\end{question}

Keeping the results in Theorems~\ref{thm:H-dim-Burch-submodules}.(6) and \ref{thm:CM-dim-Burch-submodules}.(1), \ref{Dey-Kobayashi-results}.(2) and Corollary~\ref{cor:char-Gor-Burch-ideals} in mind, we expect positive answers to the following questions.

\begin{question}\label{ques:Burch-sub-quo-CM-dim}
	Let $M$ be a Burch submodule of depth $0$ of some $R$-module $L$.
	\begin{enumerate}[(1)]
		\item If $\cmdim_R(M)<\infty$, then is $R$ {\rm CM}?
		\item If $\Ext_R^n(M,R)=0$ for any three consecutive values of $n\ge \max\{\depth(R)-1,0\}$, then is $R$ Gorenstein?
	\end{enumerate}
\end{question}

The following question is a particular case of Question~\ref{ques:Burch-sub-quo-CM-dim}.
\begin{question}\label{ques:depth-mN-zero-CM-dim}
	Let $N$ be an $R$-module such that $\depth(N)=0$ and $\fm N \neq 0$.
	\begin{enumerate}[(1)]
		\item If $\cmdim_R(\fm N)<\infty$, then is $R$ {\rm CM}?
		\item If $\Ext_R^n(\fm N,R)=0$ for any three consecutive values of $n\ge \max\{\depth(R)-1,0\}$, then is $R$ Gorenstein?
	\end{enumerate}
\end{question}

Theorem~\ref{thm:CM-dim-Burch-submodules}.(1) and Corollaries~\ref{cor:char-Gor-Burch-ideals} and \ref{cor:characterizations-via-mN}.(2) provide positive answers to Questions~\ref{ques:Burch-sub-quo-CM-dim} and \ref{ques:depth-mN-zero-CM-dim} under the conditions `$L$ is free' and `$N$ is a submodule of a free $R$-module' respectively.

\section*{Acknowledgments}
Ghosh was supported by Start-up Research Grant (SRG) from SERB, DST, Govt.~of India with the Grant No SRG/2020/000597. Saha was supported by Junior Research Fellowship (JRF) from UGC, MHRD, Govt.\,of India.

\end{document}